\documentclass[11pt]{article}

\pdfoutput=1

\usepackage{graphicx,pstricks,xypic,comment}
\usepackage{color}
\usepackage{mathrsfs} % calligraphiques en mode math que l'on appelle
\usepackage{amssymb,amscd,amsmath,amsfonts,amsthm}

\theoremstyle{plain}

\newtheorem{theo}{Theorem}[section]
\newtheorem{lem}[theo]{Lemma}
\newtheorem*{lem*}{Lemma}
\newtheorem{coro}[theo]{Corollary}

\theoremstyle{remark}

\theoremstyle{definition}

\newcommand{\ga}{\gamma}
\newcommand{\Si}{\Sigma} 
\newcommand{\te}{\theta}
\newcommand{\al}{\alpha} 
\newcommand{\De}{\Delta} 
\newcommand{\si}{\sigma} 
\newcommand{\om}{\omega} 
\newcommand{\Ga}{\Gamma}

\newcommand{\tore}{{\mathbb{T}}}    % tore 
\newcommand{\Courbe}{{\mathscr{C}}} % Famille de courbes s{\'e}parantes
\newcommand{\tf}{T}                % Fonctions traces. 

\newcommand{\la}{\langle}
\newcommand{\ra}{\rangle}

\newcommand{\C}{\mathbb C}
\newcommand{\R}{\mathbb R}
\newcommand{\N}{\mathbb N}
\newcommand{\Z}{\mathbb Z}

\newcommand{\su}{\operatorname{SU}(2)}   %% SU (2)
\newcommand{\Sl}{\operatorname{Sl } (2, \C )}  %% Sl (2, C)
\newcommand{\Slr}{\operatorname{Sl } (2, \R )}  %% Sl (2, C)

\newcommand{\tr}{\mbox{\rm tr}}       %% trace
\newcommand{\trace}{\operatorname{tr}}
\newcommand{\acos}{\mbox{\rm arccos}}
\newcommand{\mat}[1]{\begin{bmatrix} #1 \end{bmatrix}}

\newcommand{\boM}{\mathcal{M}}      %% espace de module
\renewcommand{\phi}{\varphi}
\renewcommand{\tilde}{\widetilde}
\newcommand{\Mult}{\operatorname{Mult}} %% espace d'applications
\newcommand{\Sym}{\operatorname{Sym}} %% sous-espace des sym{\'e}triques. 

\newcommand{\End}{\operatorname{End}}
\newcommand{\Hom}{\operatorname{Hom}}

\newcommand{\Op}{\operatorname{Op}}
\newcommand{\hs}{\operatorname{HS}}

\newcommand{\Tf}{{\mathcal{T}}}  % Espace de fonctions traces.
\newcommand{\Kauf}{{\mathcal{K}}}  %  Module de Kauffman

\newcommand{\Ci}{{\mathcal{C}}^{\infty}}

\newcommand{\hb}{\hbar}

\newcommand{\Ah}{ - e ^{i\hb /4}} 
\newcommand{\ep}{\varepsilon}

\title{Multicurves and regular functions on the representation variety of a surface in $\su$}

\author{L. Charles and J. March{\'e} \footnote{Institut de
    Math{\'e}matiques de Jussieu (UMR 7586), Universit{\'e} Pierre et
    Marie Curie -- Paris 6, Paris, F-75005 France.} }

\date{}
\begin{document}

\maketitle

% \begin{footnotesize}
%   \noindent \textbf{Keywords :} Birkhoff normal form, resonances,
%   pseudo-differential operators, spectral asymptotics, symplectic
%   reduction, Toeplitz operators, eigenvalue cluster.\\
%   \noindent \textbf{MS Classification :}
%   58J40, %Pseudodifferential and Fourier integral operators on
%   % manifolds
%   58J50, %Spectral problems; spectral geometry; scattering theory
%   58K50, %Normal forms
%   47B35, %Toeplitz operators, Hankel operators, Wiener-Hopf operators
%   53D20, %Momentum maps; symplectic reduction
%   81S10. %Geometry and quantization, symplectic methods
% \end{footnotesize}

\begin{abstract}
Given a compact surface $\Si$, we consider the representation space
$\boM(\Sigma)=\Hom(\pi_1(\Sigma),\su)/\su$. We show that the trace
functions associated to multicurves on $\Si$ are linearly independent
as functions on $\boM(\Sigma)$. The proof relies on the Fourier
decomposition of the trace functions with respect to a torus action on
$\boM(\Si)$ associated to a pants decomposition of $\Si$. Consequently
the space of trace functions is isomorphic to the skein algebra at
$A=-1$ of the thickened surface. 
\end{abstract}

\section{Introduction}

Given a compact and oriented surface $\Sigma$, one define its representation space as the quotient $\Hom(\pi_1(\Sigma),G)/G$. For $G=\Sl,\su,\Slr$, we obtain three related and celebrated spaces. The first one is an algebraic variety which classifies semi-stable complex bundles of rank 2 over $\Sigma$ with trivial determinant: it contains the two other ones. The second space is compact and contains an open and dense subset supporting a natural symplectic form. Its quantization provides a construction of a topological quantum field theory (TQFT) which has interesting interactions with the topology of 3-manifolds. Finally the last space contains as a connected component the Teichm{\"u}ller space, that is, the moduli space of hyperbolic structures on $\Sigma$.

The purpose of this article is to study a special class of functions
on these spaces called "trace functions". Given a 1-dimensional
submanifold $\gamma$ of $\Sigma$ and a representation $\rho\in
\Hom(\pi_1(\Sigma),G)/G$ one set
$$f_{\gamma,G}([\rho])=\prod_i(-\trace(\rho(t_i)))$$ where the $t_i$'s
represent the free homotopy classes of the components of $\gamma$. 

Assume that the Euler characteristic of $\Si$ is
negative. We prove that the functions $f_{\gamma,G}$ where $\gamma$ runs
over the isotopy classes of 1-submanifolds of $\Sigma$ without
component bounding a disc are linearly independent as functions on
$\Hom(\pi_1(\Sigma),G)/G$ for $G= \su$ or $\Sl$. 
In the case $G = \Sl$, it follows that these functions form a basis of the
coordinate ring of the representation variety. Consequently the coordinate ring is isomorphic to the skein algebra of
$\Si \times [0,1]$ at $A= -1$, cf. sections \ref{sec:results} and
\ref{sec:kauffman-module} for precise statements. 

 D. Bullock showed in \cite{bul} that the last assertion is
 equivalent to proving that the skein algebra has no zero divisors. Bullock's proof rely on a delicate analysis of algebraic relations between trace functions which started in \cite{gam}. 
Our strategy is completely different and somewhat simpler: using a
pants decomposition of the surface, we define on
$\boM(\Sigma)=\Hom(\pi_1(\Sigma),\su)/\su$ an action of the torus $\tore^{\Courbe}$ where $\Courbe$ is the set of separating curves of the decomposition. 
Moreover, 1-dimensional submanifolds of $\Sigma$ are parametrized up to isotopy by their Dehn coordinates, which is a system of parameters depending on the pants decomposition. Finally, we compute the Fourier decomposition of the trace functions relatively to the action of $\tore^{\Courbe}$ and show that one can recover the geometric intersection number of two curves and more generally the full Dehn parameters of a multicurve via its Fourier decomposition. This allows us to prove our assertion. 

Our motivation to study the trace functions is the quantization of the representation
space $\boM ( \Sigma)$. First, the space of trace functions is a
Poisson algebra, the Poisson bracket being defined with the symplectic
structure of Atiyah-Bott \cite{atiyah}. Hence the skein algebra at
$A=-1$ inherits a Poisson bracket. It appears that the skein algebra
at $A=e^{-i\hb /4}$ is a deformation quantization of this Poisson
algebra. This is a consequence of the Goldman formula \cite{goldman2}
expressing the bracket of trace function, cf. \cite{bfk} and \cite{turaev}. 

Not only do we have a formal
quantization, but also a strict quantization
provided by the topological quantum field theory. Working with the
combinatorial version of TQFT \cite{bhmv} we associate to $\Si$ a
family of Hilbert spaces $V_k ( \Si)$  and to each curve $\ga$ a family
of operators $ ( \Op_k ( \ga ) : V_k ( \Si) \rightarrow V_k (
\Si))$. Then it appears that the natural symbol of this family of operators is the trace
function of $\ga$. Indeed, by \cite{mn} the asymptotic behavior of $\Op_k (
\ga)$ as the level $k$ tends to infinity is controlled at first order
by the trace function. Furthermore the composition and the commutator
of operators corresponds to the product and Poisson bracket of the trace
functions. So the relation between the curve operators and the trace
functions is similar to the one in microlocal analysis between operators
and their symbol.  In this point of view, the interest of our result
is to produce non-vanishing trace functions and consequently
non-vanishing curve operators. This has non-trivial consequences like the
asymptotic faithfulness of the representation of the mapping class
group on $V_k (\Si)$ provided by TQFT, cf. \cite{andersen,fww,mn}.  

In the next parts of this introduction, we give precise statement of our
results (sections \ref{sec:results}
and \ref{sec:kauffman-module}) and on the quantization of $\boM (
\Si)$ (sections \ref{sec:poisson} and \ref{sec:tqft}). Note that we do not use
skein modules and TQFT in the sequel of the paper. 
 The section 2 introduces the Dehn coordinates for multicurves on a surface while section 3 describe torus actions on $\Hom(\pi_1(\Si,\su)/\su$ which are part of the action-angle coordinates of \cite{jeffrey}. In Section 4, we give the main ingredients for computing the Fourier coefficients of the trace functions. Section 5 describes the applications of the preceding computations while Section 6 explains the isomorphism of the trace functions algebra with the coordinate ring of the representation variety in $\Sl$.
 
\subsection{Results} \label{sec:results}

Let $G$ be the group $\su $ or $\Sl$. Consider the space $\Hom ( \pi,
G ) / G$ of morphisms from a group $\pi$ to $G$ up to conjugation. For any $t \in
\pi$, introduce the function $\chi_t$ of $\Hom ( \pi, G ) / G $
defined by  
$$ \chi_{t}( [\rho]) =- \trace ( \rho (t)), \qquad \rho \in \Hom ( \pi,
G)$$ 
Let us call a {\em trace function} any linear combination of the $\chi_{t}$'s
with complex coefficients. Because of the relation 
\begin{gather} \label{eq:trace}
  \trace (a) \trace (b) = \trace ( ab)  + \trace ( a^{-1} b) ,
\qquad \forall \; a, b  \in G 
\end{gather}
the space $\Tf ( \pi , G)$  of trace functions is a subalgebra of the
algebra of complex valued functions of $ \Hom ( \pi, G ) / G $.

When $\pi$ is of finite type, $\Tf ( \pi , \Sl )$ is
finitely generated and is the coordinate ring of the
$\Sl$ character variety of $\pi$. The relation with the usual definition is
explained in part \ref{variety}. 

Assume now that $\pi$ is the fundamental group of a manifold $M$. Let us
define a trace function $f_{\ga, G}$ for any isotopy class $\ga$ of
1-dimensional compact submanifold of $M$ without arc components. 
Let $t_1$, \ldots, $t_n$ be the homotopy classes of loops which are
freely homotopic to the connected component of a representant of $\ga$
and set 
$$  f_{\ga,G} = \chi_{t_1} ... \chi_{t_n}$$ 
Since the $t_i$'s are well-defined up to conjugation and inversion, this product only depends on
the isotopy class $\ga$. 

We will use these definitions in two cases. First, $M$ is a surface and $\ga$ an isotopy
class of {\em multicurve}, that is no component of $\ga$ bounds a
disc. Second, $M$ is 3-dimensional, the 1-dimensional compact
submanifolds without arc component are then called the {\em links} of $M$. 

\begin{theo} \label{the:result}
Let $G$ be  $\su$ or $\Sl$ and $\pi$ be the fundamental group of a
compact orientable surface $\Si$ with boundary (possibly empty) and negative Euler characteristic. Then the trace functions $f_{\ga, G}$, where $\ga$
runs over the isotopy classes of multicurves of $\Si$ which do not meet the boundary, form a basis of $\Tf ( \pi , G)$.
\end{theo}

The fact that the trace functions $(f_{\ga, G})$ are a spanning set follows easily
from the relations (\ref{eq:trace}). The proof that they are linearly
independent is much more delicate and is the content of theorem \ref{independance}.
Let $j$ be the obvious map  $$\Hom ( \pi, \su ) / \su \rightarrow \Hom
( \pi , \Sl )/ \Sl.$$
By the previous theorem, $j^*$ maps $\Tf ( \pi
, \Sl)$ bijectively onto $\Tf ( \pi, \su)$.

\subsection{The Kauffman module at $A=-1$} \label{sec:kauffman-module}

Let $M$ be a 3-dimensional oriented compact manifold. The skein
module of $M$ at $A=-1$ is defined as the complex vector space
freely generated by the set of isotopy classes of links of $M$ quotiented by
the relations of Figure \ref{fig:kauffman} for $A=-1$.

We denote by $\Kauf (M, -1)$ this vector space. As a consequence of the relation (\ref{eq:trace}), we have a
well-defined map $$\Kauf (M ,-1) \rightarrow \Tf ( \pi  , G) $$
sending a link $\ga$ to $f_{\ga, G}$. Here $\pi$ is the fundamental
group of $M$. 

Assume that $M = \Si \times [0,1]$ where $\Si$ is an orientable compact
surface. Then by identifying $\Si$ with $\Si \times \{ 1/2 \}
\subset M$, each multicurve of $\Si$ defines a link in $M$. By a consequence of the Reidemeister theorem, the family of isotopy
classes of multicurves of $\Si$ is a basis of $\Kauf (M ,-1)$. As a
corollary of Theorem \ref{the:result}, we obtain the 

\begin{theo} \label{the:bullock}
For any compact orientable surface $\Si$ with $\chi(\Si)<0$, the natural map 
$$ \Kauf (\Si \times [ 0,1]  ,-1) \rightarrow \Tf ( \pi  , G) $$
is an isomorphism. 
\end{theo}

Bullock proved in \cite{bul} that the kernel of this map consists in the nilradical of $\Kauf(\Si\times [0,1],-1)$. Moreover, it is claimed without proof in  \cite{prz} that this nilradical is zero.

The skein module $\Kauf ( \Si \times [0,1], -1)$ is an algebra where the product $\gamma \cdot \delta$ is defined by stacking $\gamma$ over $\delta$. As a consequence of the relations (\ref{eq:trace}) this product is sent to the product of functions by
the isomorphism of Theorem \ref{the:bullock}.

\subsection{$\Tf ( \pi  , G)$ as a Poisson algebra} \label{sec:poisson}

In this subsection and the next one, $G$ is the group $\su$, $\Si$ is a closed orientable
surface and $\pi$ its fundamental group. The space $\boM = \Hom ( \pi , G) /G$ has a
natural topology such that the trace functions are
continuous. The subset $\boM^s$ consisting of classes of
irreducible representations is open and dense. By \cite{atiyah} and
\cite{goldman}, $\boM^s$ is a symplectic manifold. Furthermore Goldman
in \cite{goldman2} expressed the Poisson bracket of the trace functions of two curves
intersecting transversally as a trace function. Consequently, $\Tf (
\pi  , G)$ is a Poisson subalgebra of $\Ci ( \boM^s, \C)$.  
The Poisson bracket appears also naturally on the topological side in the
following way. 

Let us introduce the Kauffman module $\Kauf (M, \Ah )$ of a
3-dimensional compact oriented manifold. It is defined as the free $\C [[ \hb ]]$-module generated by the set of isotopy classes
of banded links of $M$ quotiented by the relations of Figure \ref{fig:kauffman} with $A=-e^{i\hbar/4}$.

\begin{figure}[width=8cm,height=3cm]
\begin{center}
\begin{pspicture}(-2,0)(3,3)
\includegraphics{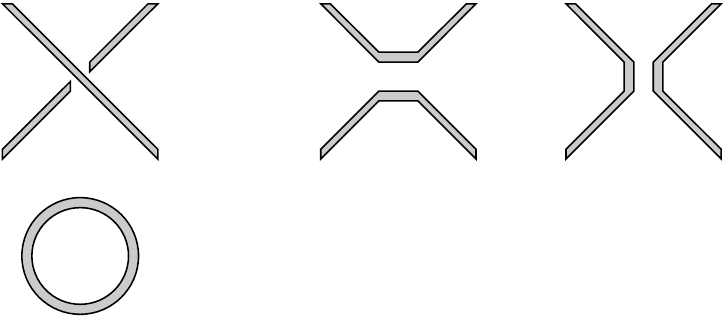}
\put(-5.2,2.3){$=A$}
\rput(-1.9,2.5){$+A^{-1}$}
\rput(-4,0.7){$=(-A^2-A^{-2})\quad \emptyset$}
\end{pspicture}
\caption{Kauffman relations}\label{fig:kauffman}
\end{center}
\end{figure}

One has a natural $\C$-linear map from $\Kauf (M, \Ah )$ to $\Kauf(M,
-1)$ sending a linear combination $\sum c_i(\hb) \ga_i$
of banded links to $\sum c_i(0) \tilde{\ga}_i$, where $\tilde{\ga}_i$
is the core of $\ga_i$. The kernel of this map is $\hb \Kauf (M, \Ah )$. 

Assume now that $M = \Si \times [0,1]$. Using the isomorphism of
theorem \ref{the:bullock},
 we obtain an exact sequence
\begin{gather}   \label{eq:suite_exacte}
0 \rightarrow \hb \; \Kauf ( M, \Ah ) \rightarrow  \Kauf (M,
\Ah ) \stackrel{\si}{\longrightarrow} \Tf ( \pi, G) \rightarrow 0  
\end{gather} 
Furthermore the
multiplication by $\hb$ is a bijection 
\begin{gather} \label{eq:mult}
   \Kauf ( M, \Ah )
\rightarrow  \hb \; \Kauf (M, \Ah )
\end{gather}
The injectiveness follows from the fact that $\Kauf (M,\Ah )$ is free as a  $\C[[\hb]]$-module, a basis being given by the family of isotopy classes of
multicurves of $\Si$.

Finally, $\Kauf ( M, \Ah )$ has a natural stacking product $*$ and a natural involution $^-$ defined by conjugating coefficients and sending a banded link $\gamma$ to $S(\gamma)$ where $S:\Sigma\times[0,1]\to\Sigma\times[0,1]$ is defined by $S(x,t)=(x,1-t)$. One has $ \overline {f * g}  = \overline {g} * \overline{f}$.

\begin{theo} \label{the:Def}
Let $G$ be the group $\su$, $\Si$ a closed surface, $\pi$ its fundamental
  group and $M = \Si \times [0,1]$. Then the Kauffman module $\Kauf
  (M, \Ah )$ with its product $*$ and involution $^-$ is a
  deformation of the Poisson algebra $ \Tf (
\pi, G) $ in the sense that for any $f, g \in \Kauf (M, \Ah )$
one has 
\begin{gather*}  \si (f * g) = \si (f) \si (g) , \qquad \si ( \overline {f} ) =
\overline { \si (f)} \\
  \si ( \hb^{-1} ( f * g - g * f ) ) = \frac{1}{i} \{ \si (f), \si (g) \} 
\end{gather*} 
where $\si$ is defined in (\ref{eq:suite_exacte}) and
$\hb^{-1}$ is the inverse of the map (\ref{eq:mult}).
\end{theo}

The first equation follows from the fact that the isomorphism of
theorem \ref{the:bullock} is an algebra morphism. For the second
equation, observe that the classes in $\Kauf ( M , \Ah )$ of the 
multicurves of $\Si$ are real and that their trace functions are real
too. Last formula on the commutator follows from the Goldman formula
expressing the Poisson bracket of two trace functions (see
\cite{goldman2}). The Poisson bracket depends on the symplectic form
which itself depends on a choice of an invariant scalar product on the
Lie algebra. We set $\langle A,B\rangle =\tr(AB^*)$ for $A,B\in
\mathfrak{su} (2)$. The symplectic form on $\Hom(\pi ,G)/G$ we are
dealing with is the symplectic reduction of the form
$\omega(a,b)=\int_{\Sigma}\langle a\wedge b\rangle$ for $a,b\in
\Omega^1(\Sigma,\mathfrak{su} (2))$.

\subsection{Topological Quantum Field Theory} \label{sec:tqft}

Let $\C [ A^{\pm 1} ]$ be the ring of Laurent polynomials and
$\Kauf (M , A)$ be the $\C[A^{\pm 1}]$-module freely generated by the
isotopy classes of banded links in $M$ quotiented by the
relations of Figure \ref{fig:kauffman}.

By sending $ \sum P_i (A) \ga_i$ to $\sum P_i
(\Ah ) \ga_i$, one identifies the skein module $\Kauf (M, A)$
with a complex subalgebra of $\Kauf (M, \Ah )$ preserved by the
involution $^-$. 

For any integer $k \geqslant 2$, the topological quantum field theories (TQFT) coming from Chern-Simons theory with gauge group $\su$ provides  a complex Hermitian space
$V_k ( \Si )$ of finite dimension. In the geometric framework, the integer $k-2$ is called the level while in the combinatorial framework constructed in \cite{bhmv}, the authors use the parameter $p=2k$.
Following them, one has a natural algebra morphism 
$$ \Kauf (M, \zeta_k) \rightarrow \End ( V_k ( \Si))$$
where $ \zeta_k = - e ^{i \pi
  /2k}$ and $ \Kauf (M, \zeta_k) := \Kauf (M, A) \otimes _{ A= \zeta_k} \C$.

For any $f \in \Kauf (M, A)$ and any $k \geqslant 2$, we denote by
$\Op _k ( f) $ the corresponding operator of $V_k (\Si)$. 
\begin{theo} \label{the:TQFT}
For any $f, g \in \Kauf (M, A)$, one has
$$ \Op _k (f * g ) = \Op_k (f) \Op_k (g), \qquad \Op_k ( \bar{f} )
= \Op_k (f ) ^* $$
and as $k \rightarrow \infty$
$$ \trace ( \Op _k ( f)) = \Bigl( \frac{k}{4\pi^2} \Bigr)^n
\int_{\boM  ^s} \si ( f) \mu + O ( k^{n-1})$$ 
where $n$ is half the dimension of $\boM^s$ and $\mu$ is the Liouville
measure. 
\end{theo}

First part follows from the general properties of TQFT. The
estimation of the trace has been proved in \cite{mn}. A factor
$4\pi^2$ appears due to the different normalization of the symplectic
form.  This estimation gives some information on the asymptotic
behavior of $ \Op _k ( f)$. 

Define the normalized Hilbert-Schmidt norm of $T \in \End (V_k( \Si
))$  by $$ \| T \| _{\hs} :=  ( \dim (V_k( \Si )))^{-1} \tr
( T T^*) .$$ 
Then by theorem  \ref{the:Def} and \ref{the:TQFT}, for any $f \in
\Kauf (M, A)$, the Hilbert-Schmidt norm of $\Op_k (f)$ is estimated by
the $L^2$ norm of $g = \si (f)$: 
 $$ \| \Op _k ( f) \|^2_{HS} = \bigl( \operatorname{Vol}( \boM ) \bigr)^{-1}
\int_{\boM ^s} | g|^2 \mu + O ( k^{-1}).$$
The space of sequences $ (T_k) \in \prod _{k \geqslant 2} \End (V_k(
\Si ))$ with a bounded Hilbert-Schmidt norm has a natural filtration
$  O(0) \supset O(1) \supset  O(2) \supset \ldots$ where $(T_k)$ is in $O(\ell)$ if  $ \| T_k \| _{\hs} \leqslant
Ck^{-\ell}$ for some $C$. It follows from theorems \ref{the:Def} and
\ref{the:TQFT} that this filtration corresponds to the formal one:
$$ \bigl( \Op _k ( f) \bigr)_k \in  O ( \ell ) \Leftrightarrow f \in
\hb^{\ell} \Kauf (M, \Ah ), $$ 
for any $f \in
\Kauf (M, A)$. In particular, if $f \neq 0$, $ \Op _k ( f)$ does
not vanish when $k$ is sufficiently large.

Using this last observation, one may deduce that the action on $V_k (
\Si)$ of an
homeomorphism of $\Si$ provided by TQFT is asymptotically non trivial
as soon as the action of this homeomorphism on the Kauffman module is
non trivial, cf. \cite{mn} for details. 

% Furthermore)) if this is satisfied, one has 
% $$ \| \Op _k ( f) \|^2_{HS} = \Bigl( \frac{2 \pi}{k} \Bigr) ^{2 \ell} \bigl( \operatorname{Vol}( \boM ) \bigr)^{-1}
% \int_{\boM ^s} | g|^2 \mu + O ( k^{-1})$$
% where $g = \si ( \hb^{-\ell} f)$.  

\section{Dehn coordinates} \label{sec:dehn-coordinates}

Given $\Sigma$ an oriented compact surface with boundary, we will call {\it multicurve} a submanifold of $\Sigma$ which do not meet the boundary and has no component bounding a disc. The Dehn theorem classifies isotopy classes of multicurves by decomposing them into simple pieces: in these simple pieces, we will allow multicurves to meet the boundary transversally as we describe it now.

Let $w$ be the arc of points in $S^1$ whose angles are in $[-\pi/8,\pi/8]$. For every $m\in \N$ we fix a subset $W_m$ of size $m$ in $w$ invariant by complex conjugation.
Let $A=S^1\times [0,1]$ be an annulus and introduce a family $A(m,t)$ of multicurves of $A$ indexed by $\N\times \Z$. If $m$ is positive, $A(m,t)$ is 
the multicurve in $A$ (unique up to isotopy) whose projection on $[0,1]$ is a submersion, which intersects each boundary circle in $W_m$ and whose algebraic intersection with the curve $\{-1\}\times[0,1]$ equals $t$ (oriented from $0$ to $1$). An example is shown in the left hand side of Figure \ref{multi}. In the case
where $m=0$,  $A(m,t)$ consists of $t$ parallel copies of the
boundary. 

Let $T$ be the surface $\bigl \{z\in \C,\text{ s.t. }|z|\le 1\text{
  and }|z\pm\frac{1}{2}|>\frac{1}{4} \bigr\}.$ Choose identifications of the boundary circles of $T$ with $S^1$ such that $1\in S^1$ is identified respectively with $p_1=\frac{1}{4}$, $p_2=\frac{-1}{4}$ and $p_3=i$. We call $T$ the standard pair of pants (or trinion). 
Let $m_1,m_2,m_3$ be three non-negative integers with even
sum. Consider a multicurve as in the right hand side of Figure \ref{multi} which
intersects the boundary circles in $W_{m_1},W_{m_2},W_{m_3}$ via the identifications with $S^1$. Denote it by $C(m_1,m_2,m_3)$.

\begin{figure}[width=5cm,height=5cm]\label{multi}
\begin{center}
\begin{pspicture}(0,0)(0,8)
\includegraphics{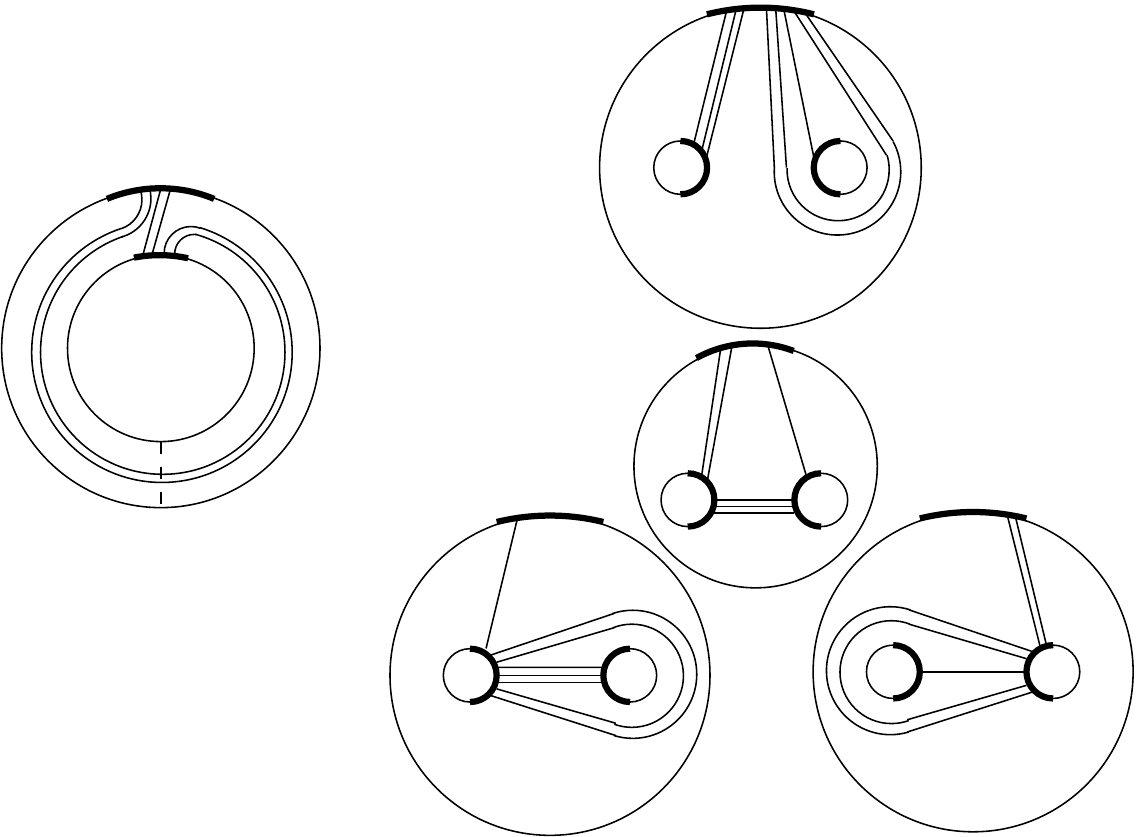}
\put(-10.5,2.7){$A(4,2)$}
\put(-6.7,3.9){$C(4,5,3)$}
\put(-8.7,0.2){$C(3,8,1)$}
\put(-4.3,0.2){$C(7,1,2)$}
\put(-6.5,5.3){$C(1,3,8)$}
\end{pspicture}
\caption{Examples of elementary multicurves}
\end{center}
\end{figure}
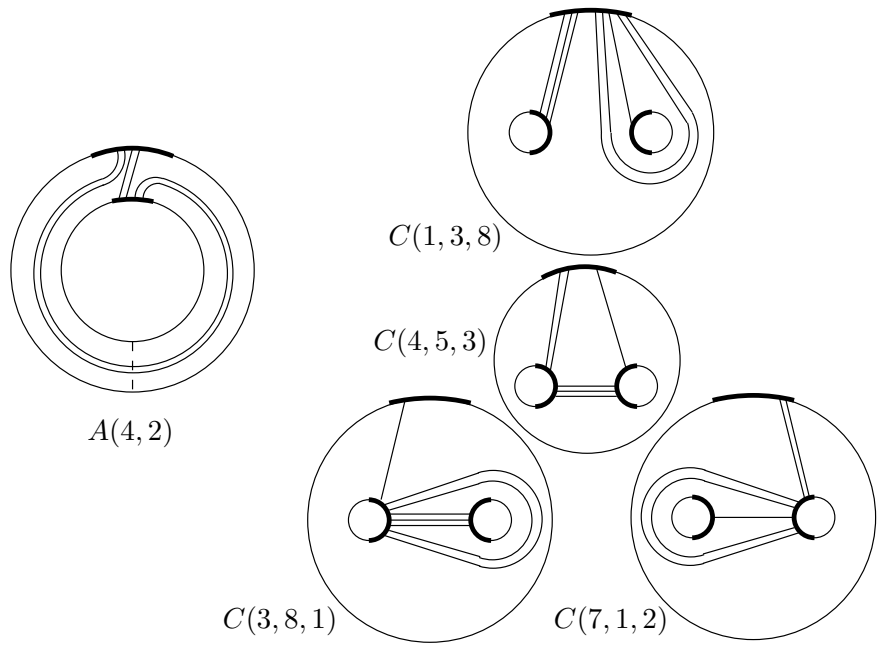

Consider now a general surface $\Sigma$ with negative Euler characteristic. From the
classification of surfaces, it appears that $\Sigma$ is obtained by
gluing trinions on boundary components. We will call {\it pant
  decomposition} of $\Sigma$ an homeomorphism 
  $$ \Phi : \Sigma
\rightarrow \Bigl( \bigcup_{i\in I}T_i \cup \bigcup_{j\in J} A_j\Bigr) /\phi$$
 where the $T_i$'s
are copies of $T$, the $A_j$'s are copies of $A$ and $\phi$ is a collection of homeomorphisms reversing the orientation between
boundary components of the $T_i$'s and the $A_j$'s. We ask that these
homeomorphisms reduce to either the identity or the complex
conjugation via the identifications of the boundary components with
$S^1$, so that they preserve the subsets $W_n$. We also ask that each boundary component of a copy of $T$ is glued to some copy of $A$. In that way, the components of the boundary of $\Sigma$ correspond to some copies of $A$.

We associate to a pants decomposition of $\Sigma$ the following graph $G$. Its vertex set is $I\cup \pi_0(\partial \Sigma)$. Edges are indexed by $J$ in such a way that the edge associated to the annulus $A_j$ connects the copies of $T$ or the boundary component of $\Sigma$ to which it is attached.
It happens that the vertices in $I$ are trivalent and the vertices in
$\pi_0(\partial \Sigma)$ are univalent. An edge is said {\it internal} if
it connects trivalent vertices, otherwise it is said {\it external}.

Consider now a surface with a pants decomposition with graph $G$. We will call {\it Dehn parameter} a pair  $(m,t)$ where $m$ is a map from the edges of $G$ to $\N$ and $t$ is a map from the edges of $G$ to $\Z$ satisfying the following conditions:
\begin{itemize}
\item[-] if $j_1,j_2,j_3$ are three edges incoming to the same vertex, then $m_{j_1}+m_{j_2}+m_{j_3}$ is even.
\item[-] for all edges $j$, if $m_j=0$ then $t_j\ge 0$.
\item[-] all external edges $j$ satisfy $m_j=0$.
\end{itemize}

Given such a Dehn parameter, one can construct a multicurve on $\Sigma$ by gluing elementary multicurves in the following way: define $C(m,t)$ as the union $\bigcup_{i}C(m^i_{1},m^i_{2},m^i_{3})\cup \bigcup _{j}A(m_j,t_j)/\phi$. In this expression,  $m^i_1,m^i_2,m^i_3$ are the values of $m$ at edges adjacent to $i$.

The classification theorem of Dehn is the following
\begin{theo}\label{dehn}

Let $\Sigma$ be a surface with pants decomposition.
The map sending a Dehn parameter $(m,t)$ to the multicurve $C(m,t)$  is a bijection between the set of Dehn parameters and the set of isotopy classes of
multicurves in $\Sigma$.
\end{theo}
We refer to \cite{flp} for the proof. In the sequel we only need to know that each multicurve of $\Si$ is isotopic to some $C(m,t)$. The fact that $C(m,t)$ is isotopic
to $C(m',t')$ only if $m=m'$ and $t = t'$ is a consequence of our
results.

\section{Moduli spaces and torus actions} \label{sec:torus-action}

Let $\Sigma$ be an oriented compact surface with boundary. 
Let $\boM ( \Si )$ be the set of isomorphism classes of pairs
$(E, \te)$ where $E$ is a flat Hermitian
bundle of rank 2 over $\Si$ and $\te $ is a flat unitary section
of $\wedge ^2 E$. An isomorphism between two pairs is an
isomorphism of flat Hermitian bundles commuting with the volume
sections. The holonomy representation of the fundamental group $\pi$
induces an isomorphism
$$ \boM ( \Si ) \rightarrow \Hom ( \pi, \su ) / \su $$
To any pants decomposition is associated a torus action on
a dense open subset of this moduli space together with a set of invariant functions which separate the
orbits.

Let $C$ be a simple curve (embedded circle) of $\Si$. Choose an orientation of $C$ and a
base point $x \in C$. Then for any $(E, \te) \in \boM (\Si)$, the
holonomy at $x$ of $E |_C$ is a unitary automorphism $g_x$
of $E_{x}$ preserving $\te_{x}$. So the trace of
$g_x$ belongs to $[-2,2]$. We set
$$ a_{C} (E, \te ) :=   \acos \bigl( \tfrac{1}{2}\tr \, (g_x)
\bigr) \in [0, \pi ].$$ 
This define a function $a_C$ on the moduli space $\boM ( \Si)$ which does not depend on the choices of
the base point and the orientation. 

Let us introduce a circle action on the
subset of $\boM ( \Si )$ consisting of bundles with a
non-central holonomy along $C$.  Let $t
\in \tore := \R / \Z $. To define $t. (E,\te)$, we assume that $C$ is
oriented. Then $E |_C$ is the
direct sum of two subbundles $E^{\pm}$, whose fibers are the eigenspaces of the
holonomy along $C$ with eigenvalues $\exp( \pm i a_{C} (E))$. Let $R_t$
be the automorphism of $E |_C$ which acts on $E^{\pm}$ by
multiplication by $\exp ( \pm 2i \pi t)$.

Next consider the surface $\Si'$ obtained from $\Si$ by cutting out
$C$. Denote by $\pi$ the projection from $\Si'$ to $\Si$. Let $C_+$
and $C_-$ be the boundary components of $\Si'$ such that $ \pi$
restricts to diffeomorphisms $\pi_\pm : C_\pm \rightarrow C$ and which
respectively preserves and reverses the orientation. 
Then $t.(E, \te)$ is the quotient of $\pi^* E$ under the identification
$$ u \sim \pi_-^*. R_t .(\pi_+^* )^{-1} (u), \qquad u \in \pi^*E |_{C_+}
= \pi_+^* (E|_C) $$
This definition does not depend on the orientation of $C$. 

Let $\Courbe$ be a
set of simple disjoint curves. Then the corresponding actions commute and
we obtain an action of the torus $\tore^\Courbe$ on the set  $\boM ^\circ
( \Sigma ) $ of bundles  
 with non-central holonomy along any curve of $\Courbe$.  
Assume that the curves of $\Courbe$ cut $\Si$ into trinions. 
As in section
 \ref{sec:dehn-coordinates} we define a graph $G$ whose edges are
 naturally indexed by $J:= \Courbe \cup
\pi_0 ( \partial \Si)$.  Introduce the map $$a : \boM ( \Si )
\rightarrow \R^{J}$$ whose coordinates are the $a_{C}$'s.

\begin{theo} \label{theo:image_moment}
The image of $a$ is the polyhedron $\De$ consisting of the $ (\al_j) \in
\R^{J}$ such that for any trivalent vertex $v$ of $G$ 
$$ | \al_{i} - \al_{j} | \leqslant \al_{k} \leqslant \min (
\al_{i} + \al_{j},  
2\pi - ( \al_{i} + \al_{j} ) ) 
$$
if $i$,$j$ and $k$ are the edges incident to $v$. 
The fibers of the restriction of $a$ to  $\boM^\circ ( \Si)$ are the
orbits of the action of $\tore^\Courbe$.
Furthermore the
action is locally free on the preimage of the interior of $\De$. 

\end{theo}

This theorem has been proved in \cite{bohr} under the assumption that $\Sigma$
has no boundary. There is no difficulty to generalize to the case where $\partial \Si \neq
\emptyset$. One shows first the result for a pair of pants $\Si$,
cf. proposition 3.1 of \cite{bohr}. The general case follows by
analyzing how one can paste flat $\su$ bundles on the trinions cut
out by $\Courbe$ to a global bundle on $\Si$.

When $\Si$ has no boundary, the open set of irreducible classes $(E,\te)$ is naturally a
symplectic manifold, as we already mentioned it in section
\ref{sec:poisson}. Then for any simple curve $C$, the function $a_C(E)$ is a moment of
the circle action corresponding to $C$, see \cite{bohr} proposition
5.4.

It is easily seen that $\Delta$ has a non-empty interior. Furthermore the
restrictions to $\operatorname{Int}  (\De)$ of the functions
$$ \exp \Bigl( i \sum_{j \in J} \ell_j a_j  \Bigr), \qquad (\ell_j) \in
\Z^{J }  $$
are linearly independent. This will be used in the proof of theorem \ref{independance}. 
 The fact that $\Delta$ has a non-empty interior will also ensure that some
Fourier coefficients do not vanish, cf. lemma \ref{lem:non-vanishing}.

\section{Fourier decomposition}

We are concern in this part on the Fourier decomposition on some trace
functions defined on the previous moduli spaces. 
Consider a disjoint union $\Ga$ of unoriented circles and a map $h : \Ga
\rightarrow \Si$. Then define the following function on $\boM ( \Si )$
$$ \tf _{h} (E,\te) = \prod _{\ga \in \pi_0 ( \Ga) } ( - \tr
( g_{\ga})) $$
where $g_\ga$ is the holonomy of $E$ along $\ga$. 
Let $\Courbe$ be a finite family of disjoint simple curves. Then for
any $k \in \Z ^ \Courbe $, the
$k$-th isotype of $\tf_h$ for the action of the torus $\tore^ \Courbe$ is
the function  
$$ \Pi_k ( \tf _h) (x)  = \int_{\tore ^{\Courbe} } \tf_h ( t.x) e^{-2
  \pi i \la t , k \ra } \; dt $$
The first  step in the computation of $ \Pi_k ( \tf _h)$ is to extend
the definition of $T_h ( E, \te)$ to the case  where $\Ga$ is a
disjoint union of circles and closed intervals.  

\subsection{Generalized holonomy trace} \label{sec:gener-holon-trace}

Let us start with an algebraic preliminary. 
Consider a finite family $(V_i)_{i \in I}$ of vector spaces. Let us
define the tensor product $\otimes_{i \in I} V_i$ assuming that the $V_i$ are odd
superspaces. First, for any
bijections $\si$ and $\si'$ from $ \{ 1, \ldots , n \}$ to $I$, one has
a commutation map 
$$ c_{\si, \si'} : V_{\si (1) } \otimes \ldots \otimes V_{\si (n)} \rightarrow V_{\si
  '  (1) } \otimes \ldots \otimes V_{\si '(n)}$$
sending $v_{1} \otimes \ldots \otimes v_{n}$ to $ ( -1) ^{\epsilon
  (\al)} v_{\al (1)} \otimes
\ldots \otimes v_{\al (n)}$ where $\al = \si ' \circ \si^{-1}$. Since
$c_{ \si', \si''} \circ c_{\si, \si'} = c_{ \si, \si''}$, one can
identify coherently these tensor products by taking the projective limit of this system, which defines the unordered
tensor product. More precisely, we define $ \otimes _{i \in I} V_i$ as
the subspace of the product  
$$ \prod_{\si \in \operatorname{bij}( \{ 1, \ldots, n \} , I)  } \bigl( V_{\si(1)} \otimes \ldots \otimes V_{\si (n)} \bigr)$$
consisting of the families $ (v_\si) $ such that $ c_{\si, \si'} (
v_{\si} ) = v_{\si '}$. 

 Let $E \rightarrow \Si$ be a flat vector bundle of rank 2 and $\te$ a
non-vanishing flat section of $\wedge ^2 E$. Observe that each
fiber $E_x$ has a complex symplectic product $\om_x$ defined in such a way
that for all $u$, $v \in E_x$,
$$ u \wedge v = \te_x \Rightarrow \om_x (u,v) = 1.$$

Consider a
1-dimensional compact manifold $\Ga$ and a continuous function $h :
\Ga \rightarrow \Si$. We will define the generalized holonomy trace of $h$
as a vector 
$$ \tf _h (E, \te) \in  \bigotimes_{ p \in \partial \Ga } E_{h(p)}   $$ 
If $\Ga$ is an interval, we choose an orientation of it and denote by $p$
and $q$ the source and target point of $\Ga$. Then the holonomy
along $\Ga$ is an homomorphism  $$ A \in \Hom (E_{h(p)} , E_{h (q)}) \simeq
  E_{h(q)} \otimes E^*_{h(p)} .$$ 
Identifying $E_{h(p)}$ with
  $E^*_{h(p)}$ by the map sending $u$ to $\om_{h(p)} (u, \cdot)$, we
  obtain a vector $A_\flat \in E_{h(q)} \otimes E_{h(p)}$. We set $$\tf_h
  (E, \te) := A_\flat.$$ 
  The important point is that this definition
  does not depend of the orientation if we use the previous
  identification between $ E_{h(p)} \otimes E_{h(q)}$ and $E_{h(q)}
  \otimes E_{h(p)}$. Indeed with this identification, we have that
  $(A^{-1})_\flat = A_\flat$. This is easily deduced from the fact
  that  $A$ is a linear symplectomorphism.
 If $\Ga$ is a circle, we define $\tf_{h} ( E, \te)$ as previously as
the opposite of the trace of the holonomy of $h$ (the super-trace). 
Finally, if $\Ga$ has several components $(\Ga_i)$, we define $\tf_h(E, \te)$ as
the tensor product of the $\tf_{h|_{\Ga_i}} ( E, \te)$.

Now let $h':\Ga'\to \Si$ be a continuous map and $p$ and $q$ be two distinct points in $\partial \Ga'$ such that $h'(p) = h'(q)$. Then by identifying $p$ with $q$, we get a compact 1-manifold $\Ga$ with a map $h :\Ga \rightarrow \Si$. It is
easily checked that 
$$ \tf_{h} (E, \te) = C(\tf_{h'} (  E, \te))$$ where $C$
is the contraction 
$$ \bigotimes _{r \in \partial \Ga' } E_{h(r)} \longrightarrow \bigotimes
_{r \in \partial \Ga' \setminus \{ p, q \} } E_{h(r)}$$ 
sending $ u \otimes v \otimes w $ to $u.\om (v,w)$ if $w \in
E_{h(p)}$, $v \in E_{h(q)}$ and $u \in \bigotimes\limits
_{r \in \partial \Ga' \setminus \{ p, q \} } E_{h(r)}$.

\subsection{The basic computation} \label{sec:basic-computation}

In this part we consider the case of a single simple curve $C$ of $\Si$ with its associated circle 
action. Let $\Ga$ be a finite disjoint union of circles and $h : \Ga
\rightarrow \Si$ be a map intersecting $C$ transversally. 

Let $\Ga'$ be the compact 1-manifold obtained by cutting out the points $p$ of $\Gamma$ such that  $h(p)$ is on $C$ and let $h'$ be the obvious map $\Ga' \rightarrow
\Si$. Then it follows from the considerations of section
\ref{sec:gener-holon-trace} that $\tf_h (E, \te)$ is the contraction of the generalized trace
holonomy $\tf_{h'} (E, \te)$ by the linear map $$ C: \bigotimes_{p \in \partial \Ga'} E_{h'(p)}
\rightarrow \C$$ defined as follows. For any $p \in h^{-1} (C)$, let
$j_1(p)$ and $j_2 (p)$ be the corresponding boundary points of
$\Ga'$, so that 
$$  \bigotimes_{p \in \partial \Ga'} E_{h'(p)} = \bigotimes_{p \in
  h^{-1}(C)} E_{h(p)}^{\otimes \{ j_1 (p), j_2 (p)\} } $$
Then $C$ is the map sending $\bigotimes\limits_{ p \in h^{-1} (C)} (v_{p,1}
\otimes v_{p,2})$ to $ \prod\limits_{p\in h^{-1}(C)}\om_{h(p)} ( v_{p,1} , v_{p,2}).$ 

We will compute $\Pi_k (T_h)$ in terms of these generalized holonomy traces. Assume that $C$ is oriented and pick $(E,\te)$ such that its holonomy along $C$ is not central.  For any $p\in h^{-1} (C)$, we choose
two unitary eigenvectors $e^p_\pm \in E_{h(p)}$ of the holonomy along
$C$  with corresponding eigenvalues $\exp ( \pm i a_{C} (E))$ and  such
that $$\te _{h(p)} = e^p_+ \wedge e^p_- .$$ Assume also that $j_1(p)$ and
$j_2(p)$ are chosen in such a way that the tangent vector of $\Gamma$ oriented from $j_1$ to $j_2$ followed with the oriented tangent vector of $C$ form a direct basis in $\Si$.

\begin{lem} \label{lem:basic-computation}
Let $k \in \N$. If $k$ is bigger than the cardinal of $h^{-1}(C)$,
then $\Pi_k (T_h)$ and $\Pi_{-k}(T_h)$ vanish. If $k$ is equal to the
cardinal of $h^{-1} (C)$, then 
$$ \Pi_{\pm k} ( \tf_h ) (E, \te)  = C_{\pm} ( \tf_{h'} (E, \te)),$$
with $C_{\pm}$  the linear map from $ \bigotimes\limits_{p \in
  h^{-1}(C)} E_{h(p)}^{\otimes \{ j_1 (p), j_2 (p)\} }$ to $\C$ given
by
$$ C_{\pm} \bigl ( \bigotimes _{ p \in h^{-1} (C)} (v_{p,1}
\otimes v_{p,2} ) \bigr) = \prod _{p \in h^{-1} (C)} \la v_{p,1},
e^p_{\pm} \ra \la v_{p,2}, e^p_{\mp} \ra
$$
where the bracket denotes the scalar product. If $C$ does not
intersect $h$, $C_{\pm}$ is the
identity map.
\end{lem}

\begin{proof} 
Assume first  that $h^{-1} (C)$ is reduced to a point. Then the proof is
based on the simple observation that 
\begin{gather} \label{eq:1} 
 \tf_h (t. (E, \te)) = C_t ( \tf_{h'}  (E, \te)), \qquad \forall \;
t \in \tore 
\end{gather}
where $C_t$ is the linear map from $E_{h(p)}^{\otimes \{ j_1 (p), j_2
  (p)\} }$ to $\C$ given by 
$$ C_t(v_1 \otimes v_2 ) = \om_{h(p)}  ( R_t v_1, v_2)$$
with $R_t$ the automorphism of $E_C$ entering in the definition of the
circle action. Next a straightforward computation leads to 
$$ \om_{h(p)} (R_t v_1, v_2) = e^{2i \pi t} \la v_1 , e_+^p \ra \la  v_2 ,e_-^p \ra + e^{-2i \pi t} \la v_1 , e_-^p \ra \la  v_2 , e_+^p \ra  $$
which proves the result, taking into account that the Hermitian product is $\C$-linear on the left. In the case where $h^{-1} (C)$ consists of
several points, (\ref{eq:1}) is still satisfied with the contraction 
$$ C_{t} \bigl ( \bigotimes\limits_{ p \in h^{-1} (C)} (v_{p,1}
\otimes v_{p,2} ) \bigr) = \prod _{p \in h^{-1} (C)} \om_{h(p)}  (
R_t v_{p,1}, v_{p,2} ) .$$
Replacing each term in the product by the previous formula leads to
the result.  \end{proof}

\subsection{Fractional Dehn twist}

Consider again a curve $C$ on $\Sigma$ and a map $h:\Ga\to\Si$ which is transverse to $C$ in $k$ points. One defines the {\it fractional Dehn twist} of $h$ along $C$ of order $\ell /k$ in the following way: let $\Gamma'$ be the 1-manifold $\Gamma$ obtained by cutting out the points $p$ such that $h(p)$ is on $C$. 
Let us orient $C$, then for any $p$ in $P=C\cap h(\Gamma)$, one
defines two points $j_1(p)$ and $j_2(p)$ in $\partial \Gamma'$ with
the same convention as in the section
\ref{sec:basic-computation}. Moreover, the orientation of $C$ provides
$P$ with a cyclic order, so $\Z / k \Z$ acts on $P$. One sets $$\Gamma_{\ell /k}=\Gamma'\cup
\coprod_{p \in P }[0,1]_p/\sim$$ where we identify $0_p$ with
$j_1(p)$ and $1_p$ with $j_2(p+\ell)$.

Choose a parameterization $\alpha:S^1 \to C$ respecting the orientation
such that $P$ corresponds to the set of $k$-th roots of unity. We
define the map $h_{\ell /k}:\Gamma_{\ell /k}\to \Sigma$ as being equal
to $h$ on $\Gamma'$ and such that for all $t\in [0,1]$ and $p\in P$
one has 
$$ h_{\ell /k}(t_p)=\alpha \bigl( \exp ( \tfrac{2i\pi \ell}{k}  t_p ) . \alpha^{-1} (p) \bigr) .$$
The manifold $\Gamma_{\ell /k}$ and the map $h_{\ell /k}$ do not
depend up to homotopy on the orientation of $C$ nor on the
parameterization $\alpha$.

%Identify $C$ with $S^1$ such that $P$ lies in the standard window $w$. Then, one defines a new surface $\Si_{\ell /k}$ homeomorphic to $\Si$ by cutting $\Sigma$ along $C$ and inserting the multicurve $A(k,\ell )$. There is a unique map $h_{\ell /k}:\Gamma_{\ell /k}\to \Sigma_{k,\ell }\simeq \Sigma$ which extends the previous one. We call it the fractional Dehn twist of $h$ along $C$ of order $\ell /k$. It does not depend on the orientation of $C$. 

\begin{lem} \label{lem:fractional-dehn-twist}
One has  $ \Pi_{\pm k}  ( \tf_{h_{\ell /k}} ) = (-1)^{\ell  ( k-1)} e^{ \pm i \ell  a_C}   \Pi_{\pm k}  ( \tf_{h} )$ 
\end{lem}

\begin{proof}
This is an easy application of the lemma \ref{lem:basic-computation} using the explicit description of the fractional Dehn twist above. More precisely, denote by $H_{\ell /k}$ the restriction of $h_{\ell /k}$ to $\coprod_{i=1}^k[0,1]_i$ and choose a parameterization $\alpha:S^1\to C$ as above. Given a bundle $E$ in $\boM(\Sigma)$ one can suppose that it is trivialized on $P$ in such a way that the holonomy from $p$ to $p+1$ is 
$$\mat{\exp(2i\pi a_C/k)& 0\\ 0 &\exp(-2i\pi a_C/k)}.$$
Then $H_{\ell /k}$ and $H_{0/k}$ differ by the factor $\exp(2i\pi a_C/k)^{k\ell }$ and the sign of the permutation sending $x$ to $x+\ell $ in $\Z/k\Z$.
\end{proof}

\subsection{Non vanishing of some isotypes} \label{sec:non-vanishing}

Consider a pants decomposition of $\Si$ in the sense of section
\ref{sec:dehn-coordinates} with graph $G$. Let $S^1 \times \{ 1/2\}$ be the core of the standard annulus $S^1
\times [0,1]$.  For each internal edges of $G$, we consider the core
of the corresponding annulus and the circle action it generates. This
defines a torus action
of $\tore ^{\Courbe}$ where $\Courbe$ is the set of internal
edges. 

Consider the multicurve $C(m,t)$ with Dehn coordinates $(m,t)$. By the
first part of Lemma \ref{lem:basic-computation}, the $k$-th isotype of
$T_{C(m,t)}$ vanishes if $|k_j| > m_j$ for some internal edge $j$.

\begin{lem} \label{lem:non-vanishing}
Let $k \in \Z^{\Courbe} $ be such that $|k_j| = m_j$ for all $j \in
\Courbe$. Then the function $\Pi_k (\tf _{C(m,t)})$ does not vanish on $a^{-1} (
  \operatorname{Int} (\De) ) $, where $a$ is the map introduced in
  section \ref{sec:torus-action} and $\Delta$ is the image of $ a$. 
\end{lem}

\begin{proof} 
By lemma \ref{lem:fractional-dehn-twist}, one may assume that $t$
vanishes. One can compute $\Pi_k (\tf _{C(m,0)})$ in exactly the same
way as we did in section \ref{sec:basic-computation}. The cut manifold $\Ga'$ is now a disjoint union of closed intervals with a map $h':\Ga'\to \Si$. 
For any oriented curve $C$ in $\Courbe$ and  point $p$ of $C$, one
considers a unitary eigenvector $e$ of $E_p$ for the holonomy along
$C$ whose eigenvalue has positive imaginary part. To any point $p$ in $\partial\Ga'$ such that $h'(p)\in C$, one associates the vector $e^p_+$ (resp. $e^p_-$) corresponding to the orientation of $C$ going to the left from $p$ (resp. to the right).

By generalizing lemma \ref{lem:basic-computation}, one obtains that 
$$ \Pi_{k} ( \tf_h ) (E, \te)  = C_{k} ( \tf_{h'} (E, \te)),$$
with $C_{k}$  a linear map from $ \bigotimes_{p \in
  \partial \Ga'} E_{h'(p)}$ to $\C$ of the
form 
$$ C_{k} \bigl ( \otimes _{ p \in \partial \Ga'} v_{p}  \bigr) =
 \prod _{p \in \partial \Ga'} \la v_{p}, e^p_{\epsilon(k,p)} \ra \mod \tore
$$ 
where $\epsilon(k,p)$ is the sign of $k_j$ if $p$ belongs to $j\in \Courbe$.
Since $\tf_{h'} (E, \te) $ is the tensor
product of the $\tf_{h'|_\ga} (E, \te)$ where $\ga$ runs overs the components of
$\Ga'$, one can rewrite the previous product as 
$$ \Pi_{k} ( \tf_h ) (E, \te)  = \prod _{\ga \in \pi_0 ( \Ga')} C^{\ga}_k
( \tf_{h'|_\ga} (E, \te)) \mod \tore$$
with $C^{\ga}_k(v_p \otimes v_q) = \la v_{p}, e^p_{\epsilon(k,p)} \ra \la v_{q}, e^q_{\epsilon(k,q)} \ra
$ if the source and target of $\ga$ are $p$ and $q$. 

To prove the lemma, it suffices to show that the factors $ C^{\ga}_k ( \tf_{h'|_\ga} (E, \te))$ do not vanish when $a (E, \te)$ belongs to
the interior of $\De$. There are two cases to consider, according to whether 
the endpoints of $\ga$ belong to the same separating curve or to
different ones, cf. for example $C(2,0,0)$ and $C(1,1,0)$. 

In the
second case, consider the pair of pants containing $\ga$ and let
$j_1$, $j_2$ and $j_3$ be the bounding curves with the induced
orientation. Assume that the
endpoints $p$ and $q$ of $\gamma$ belongs to $j_1$ and $j_2$ respectively. If $ C^{\ga}_k ( \tf_{h'|_\ga} (E, \te))$ vanishes, one
deduces from the definition of $\tf_{h'|_\ga} (E, \te)$ that the
holonomy along the loop $h' |_\ga$ sends the decomposition $\C e^p_{+}
\oplus \C e^p_{-}$ to the decomposition $\C e^q_{+} \oplus \C
e^q_{-}$, permuting possibly the summands. Choosing $p$ as a base point, one obtains that the
holonomies along $j_1$ and the concatenation $\ga^{-1}  j_2 \ga$
commute. Observe furthermore that $j_3^{-1}$ is freely isotopic to $\gamma^{-1}j_2\gamma j_1$. It follows that for some $\ep_i \in \{ 1, -1 \}$, the three eigenvalues $\exp(\ep_1 i a_{j_1}(E)),\exp(\ep_2 i a_{j_2}(E)),\exp(\ep_3 i a_{j_3}(E))$ have product equal to 1. This implies that 
$$ \ep_1 a_{j_1} (E ) + \ep_2 a_{j_2} (E) + \ep_3 a_{j_3} (E) \equiv 0 \mod 2 \pi
\Z$$  
This formula is only satisfied if $a (E, \te)$
belongs to the boundary of $\De$.

In the first case, $h'|_\ga$ connects two points $p$ and $q$ of the same
boundary circle $j_1$ by going around another boundary circle
$j_2$. Since the neighborhoods of the endpoints of $\ga$ are on the
same side of $j_1$, the eigenvalues of $e_p$ and $e_q$ are the
same. Using this, one shows that if $C^{\ga}_k(\tf_{h'|_\ga} (E, \te))$ vanishes,
then the holonomy along a path joining $j_1$ to $j_2$ preserves the
decomposition into eigenspaces. In other words, we are again in the
second case. 
\end{proof}

\section{Consequences} 

\subsection{Geometric intersection numbers} 

The
{\em geometric intersection number} of two isotopy classes of
multicurves $\xi$ and $\eta$ is the minimal number
of intersection points of a representant of $\xi$ with a representant
of $\eta$. 

If $\xi$ has only one connected component, it
generates a circle action on the  $\boM ( \Si )$. One may
characterize the geometric intersection number of $\xi$ and $\eta$ in
terms of this action and the trace function of $\eta$. 

\begin{theo} 
For any isotopy classes $\xi$ and $\eta$ of a curve and a multicurve
res\-pec\-ti\-ve\-ly, the geometric intersection number of $\xi$ and
$\eta$ is the biggest $k$ such that the $k$-th isotype of the trace function of
$\eta$ with respect to the circle action generated by $\xi$ does not vanish.
\end{theo}

\begin{proof} 
If $k$ is larger than the geometric intersection number, then the
$k$-th isotype vanishes by Lemma
\ref{lem:basic-computation}. Conversely, assume that $\xi$ does not
bound a disc and is not parallel to a boundary component, otherwise the
result is trivially satisfied. Then there is a pants decomposition
with a separating curve $C$ representing $\xi$. By Theorem
\ref{dehn}, $\eta$ may be represented by a Dehn multicurve
$C(m,t)$. Then lemma \ref{lem:non-vanishing} shows that the
$k$-isotype of $\tf_{C(m,t)}$ does not vanish if $k$ is the number of
intersection points of $C(m,t)$ with $C$.  
\end{proof}
The proof shows also that the intersection number is realized when the
representant of $\xi$ is a curve of a pants decomposition and the
representant of $\eta$ is one of the associated Dehn curves, a
well-known result. 

The following corollary has been proved in \cite{goldman2}. Denote by $\boM^s ( \Si)$  the subset of $\boM
( \Si)$ consisting of the irreducible bundles.

\begin{coro} Assume that $\Si$ is closed so that $\boM^s ( \Si)$
  is a symplectic manifold. Let $\xi$ be a curve and $\eta$ a multicurve in $\Si$. Then the Poisson bracket of the trace functions of $\xi$ and $\eta$ vanishes if and only if $\xi$ and $\eta$
admit non-intersecting representants. 
\end{coro}

This follows from the fact that the function $\arccos (- \frac{1}{2}
T_\xi)$ is a moment of the circle action generated by $\xi$.

\subsection{Independence of trace functions} 

One can now prove the following theorem:

\begin{theo}\label{independance}
The functions $T_{\xi}$, where $\xi$ runs over isotopy classes of
multicurves in $\Si$, are linearly independent.
\end{theo}

Since the isomorphism between $\boM (\Si)$ and $\Hom ( \pi, \su) /
\su$ identifies the function $T_\xi$ with the trace function
$f_{[\xi], \su}$, we deduce theorem \ref{the:result} for the group
$\su$. Consider the map $j$ 
$$ \Hom ( \pi, \su) /
\su \rightarrow \Hom ( \pi, \Sl) /
\Sl $$
Then $j^* f_{ \ga, \Sl} = f_{\ga, \su}$. So the independence of the $
f_{ \ga, \su}$ implies the independence of the $f_{\ga, \Sl}$, which
proves theorem \ref{the:result} for the group
$\Sl$.

\begin{proof}
Consider a pants decomposition of $\Si$ with associated graph $G$ and
the action of the torus $\tore^{\Courbe}$ as in section
\ref{sec:non-vanishing}. Let $J = \Courbe \cup \pi_0 ( \Si)$ be the
set of edges of $G$. By theorem \ref{dehn}, one may suppose that $\xi$ runs over the
multicurves $C(m,t)$ for admissible maps $m: J \to \N$ and $t: J
\to\Z$. 

 Let $\beta$ be a vanishing linear
combination of the $T_{\xi}$'s. Let $\beta=\sum_m \beta_m$ be its
decomposition with respect to the multi-degree $m$. Write
$$\beta_{m}=\sum_t \lambda_{m,t} T_{C(m,t)} .$$ 
Consider an arbitrary order on $J$ and the corresponding
lexicographical order on maps $m$.
Let $M$ be the maximal $m$ with a non-vanishing family of coefficients $(\lambda_{m,t})_t$. By
lemma \ref{lem:basic-computation}, for $m<M$ one has
$\Pi_{M}(\beta_m)=0$. Hence $\Pi_{M}(\beta_M)=\Pi_{M}(\beta) = 0.$
Furthermore, by lemma \ref{lem:fractional-dehn-twist}, 
\begin{alignat*}{2} 
\Pi_M(\beta_{M})= & \sum_{t} \lambda_{M,t} \Pi_M(T_{C(M,t)}) \\
= & \biggl( \sum_t \lambda_{M,t} \prod_{j \in J} \varphi_{M_j,t_j} (a_j) \biggr)
 \Pi_M(T_{C(M,0)})
\end{alignat*}
where for any integers $k, \ell$ and real $\al$, 
$$ \varphi_{k, \ell}  (\al) = \begin{cases} (-1) ^{\ell ( k-1)} \exp
  (i \ell \al) \text{ if } k \neq 0 \\ (2 \cos ( \al ))^{ \ell }
  \text{ otherwise} \end{cases}
$$
By lemma \ref{lem:non-vanishing}, the function $\Pi_M(T_{C(M,0)})$ does
not vanish on $a^{-1} ( \operatorname{Int}  \De)$. Furthermore, by the remark after
theorem \ref{theo:image_moment}, the functions $$\prod_{j \in J}
\varphi_{M_j, t_j} (a_j)$$ where $t$ runs over the maps $J \rightarrow
\Z$ such that $(M,t)$ are Dehn coordinates, are linearly independent
over $a^{-1}( \operatorname{Int} \De)$. So
the coefficients $\lambda_{M,t}$ vanish, contradicting the maximality of $M$.
\end{proof}

\section{On the character variety} \label{variety}

Denote by $V$ the vector space $\C^2$ and by $G$ the group
$\operatorname{Sl } (2, \C ) \subset \End (V)$. 
Let $\pi$ be a finitely generated group. Choosing generators $t_1, \ldots ,
t_n$ of $\pi$, we identify the set $\Hom ( \pi , G)$ of morphisms  with a closed algebraic
subset of $ (\End
V)^n$ by sending the morphism $\rho$ to $( \rho ( t_1), \ldots ,
\rho ( t_n) )$. This endows $\Hom ( \pi, G)$ with a structure of
affine variety. Its coordinate ring $\C[\Hom (\pi, G)]$ is the quotient of
$\C [ (\End V )^n ]$ by the ideal of polynomial functions vanishing on
$\Hom (\pi, G)$. 

The action by conjugation of $G$ on $\Hom (
\pi, G)$ is regular. By Hilbert theorem, the ring of invariant
functions is finitely generated. By definition $\C [ \Hom ( \pi , G) ] ^G$ is
the coordinate ring of the character variety. 

One may naturally identify this coordinate ring with a subspace of the
space  of complex valued functions on the quotient $\Hom ( \pi , G) /G$. 
By the following theorem, this subspace is the space $\Tf ( \pi, G)$
of trace functions.  
\begin{theo} 
The ring $\C [ \Hom ( \pi , G) ] ^G$ is generated as a vector space by the
functions $\chi_t$  
$$  \chi_ t ( [\rho]) = - \trace ( \rho (t)), \qquad \rho \in \Hom ( \pi,
G)$$ 
where $t$ runs over $\pi$. 
\end{theo}

We provide an elementary proof of this well-known result. 
\begin{proof}
Because of the trace relation (\ref{eq:trace}), 
it suffices to prove that the ring of regular invariant functions is generated
by the $\chi_t$'s as an algebra. Then by averaging the
action of the
compact subgroup $\operatorname{SU} (2) \subset G$, we get that any
invariant regular function is represented by an invariant polynomial
function $P \in \C [ ( \End V)^n ]^G$. 

The symplectic form of $V$
induce an isomorphism between $V$ and $V^*$ and consequently  an
isomorphism between $\End V$ and $ \End V^*$. By
composing with the transposition map, we define an equivariant isomorphism  $a \rightarrow a^*$ of $\End (V)$. 
We will prove that $ \C [ (
\End V)^n ]^G$ is generated as an algebra by the functions 
\begin{gather} \label{eq:foncgen}
 ( a_1, \ldots, a_n ) \rightarrow \trace ( w (a) ) 
\end{gather}
where $w(a)$ runs over the words in the letters $a_1, a_1^*, \ldots,
a_n, a_n^*$. Since the elements of $G$ satisfy $a ^{-1} = a^*$,
this will end the proof.  

Let $W = \End (V)$. For any $n$-tuple $(d_1, \ldots, d_n)$ of non
negative integers, the space of polynomials on $W^{\oplus
  n}$ which are homogeneous of degree $d_i$ on the $i$-th variable is
isomorphic with $ \otimes_i \Sym ^{d_i} (W_i^*)$. So 
$$\C [ W^{\oplus n} ] = \bigoplus_{d \in \N^n } \bigotimes_{i=1,
  \ldots , n } \Sym^{d_i}( W^*) \subset   \bigoplus_{d \in \N^n } \bigotimes_{i=1,
  \ldots , n } ( W ^*)^{\otimes d_i} =:E $$
A right inverse of this inclusion is the following map from $E$ to $ \C [
W^{\oplus n} ]$ 
$$ \bigotimes_{i=1,
  \ldots , n } (\ell_i^1 \otimes \ldots \otimes \ell_i^{d_i})
\rightarrow  P (a_1, \ldots , a_n ) = \prod _{i=1, \ldots , n; \atop \; j=
  1, \ldots , d_i} \ell_i^{ j} ( a_i)  $$ 
Observe that the invariant subspace of $E$ is sent onto $ \C [ W ^{\oplus n}]^G$. Using
again the identification of $V$ with $V^*$, we have that $W \simeq V
\otimes V^* \simeq V^* \otimes V^*$.  By lemma \ref{lem:inv}, the
invariant subspace of $(V ^* )^{\otimes 2m} \simeq  (V^{\otimes 2m }
)^*$ is generated by the maps 
$$ x_1 \otimes \ldots \otimes x_{2m} \rightarrow \om ( x_{\si(1)}, x_{\si(2)})
  \ldots \om ( x_{\si(2m-1)}, x_{\si(2m)}) $$
where $\si$ runs over the permutation of $\{ 1, \ldots , 2 m\}$. We
deduce the invariants subspace of $E$ and then that the functions
(\ref{eq:foncgen}) generate $\C [ W^{\oplus n} ]^G$.
\end{proof} 

Denote by $\Mult_n $ the space of multilinear maps from $V^{\times
  n}$ to $\C$ and by $\Sym_n $ the subspace of symmetric maps.
For any non-negative integer $k \leqslant n$ with the same parity as
$n$, let $\Mult_{n,k}$ be the subspace of $\Mult_n$ generated by the maps
$$ P(x_{\si(1)}, \ldots, x_{\si (k)}) \om (x_{\si(k+1)},  x_{\si
    (k+2)}) \ldots \om ( x_{\si(n-1)}, x_{\si (n)})$$
where $\si$ ranges over the permutations of $\{ 1, \ldots , n\}$ and
$P$ over $\Sym_k$.
Recall that the spaces $\Sym_n $ are the irreducible representations of
$G$.

\begin{lem} \label{lem:inv}
The decomposition into isotypical subspaces of $\Mult_n $ is 
$$ \Mult_n = \bigoplus_{ 0 \leqslant k \leqslant n, \atop \; k = n \mod 2} \Mult_{n,k}$$
\end{lem}

\begin{proof} 
$\Mult_{n,k}$ is clearly a subspace of the $k$-th isotypical component of
$\Mult_n $. Since $\Mult_n $ is the direct sum of its
isotypical components, we just have to show that the spaces $\Mult_{n,k}$
generates $\Mult_n$. This may be proved by induction over $n$ by
using that any multilinear map symmetric with respect to the
($n-1$) first arguments is of the form 
$$  L (x_1, \ldots , x_{n} ) = M(x_1, \ldots ,x_{n}) + \sum_{i=1}^n \om
  (x_i,x_n) N (x_1, \ldots , \widehat{x_i}, \ldots x_{n-1})$$
with $M \in \Sym _n $ and $N \in \Sym_{n-2} $. To show this last
fact, observe that the
map sending $(M,N)$ to $L$ is an isomorphism from $\Sym _n \oplus
\Sym_{n-2} $ onto the subspace of $\Mult_n  $
consisting of maps symmetric in the $(n-1)$ first arguments. Indeed,
this morphism is injective by $G$-equivariance and we conclude by counting dimensions. 
\end{proof}

\end{document}